\documentclass[10pt,final,twocolumn]{IEEEtran}                                                          

\newtheorem{proposition}{Proposition}
\newtheorem{theorem}{Theorem}
\newtheorem{definition}{Definition}
\newtheorem{lemma}{Lemma}
\newtheorem{corollary}{Corollary}

\usepackage{cite}

\ifCLASSINFOpdf
\usepackage[pdftex]{graphicx}
\graphicspath{{../pdf/}{../jpeg/}}
\DeclareGraphicsExtensions{.pdf,.jpeg,.png}
\else
\usepackage[dvipdf]{graphicx}
\graphicspath{{../eps/}}
\DeclareGraphicsExtensions{.eps}
\fi

\usepackage[cmex10]{amsmath}
\usepackage{amssymb}
\usepackage{array}
\usepackage{mdwmath}
\usepackage{mdwtab}
\usepackage{eqparbox}
\usepackage[tight,footnotesize]{subfigure}

\title{\LARGE \bf
Time-Optimal Frictionless Atom Cooling in Harmonic Traps
}

\author{Dionisis Stefanatos, Heinz Schaettler, and Jr-Shin Li
\thanks{This work was supported by the Air Force Office of Scientific Research under Young Investigator Award FA9550-10-1-0146}
\thanks{D. Stefanatos, H. Schaettler, and J.-S. Li are with the Department of Electrical and Systems Engineering, Washington University,
        St. Louis, MO 63130, USA
        {\tt\small dionisis@seas.wustl.edu, hms@wustl.edu, jsli@seas.wustl.edu}}%
}

\begin{document}

\maketitle

\begin{abstract}

Frictionless atom cooling in harmonic traps is formulated as a time-optimal
control problem and a synthesis of optimal controlled trajectories is obtained. This work has already been used to determine the minimum time for transition between two thermal states and to show the emergence of the third law of classical thermodynamics from quantum thermodynamics. It can also find application in the fast adiabatic-like expansion of Bose-Einstein condensates, with possible applications in atom interferometry.
This paper is based on our recently published article in {\it SIAM J. Control Optim.} \cite{Stefanatos11}.

\end{abstract}


\IEEEpeerreviewmaketitle

\section{INTRODUCTION}

During the last decades, a wealth of analytical and numerical tools from control theory and optimization have been successfully employed to analyze and control the performance of quantum mechanical systems, advancing quantum technology in areas as diverse as physical chemistry, metrology, and quantum information processing \cite{Mabuchi05}. Although measurement-based feedback control \cite{Doherty00} and the promising coherent feedback control \cite{James08} have gained considerable attention, open-loop control has been proven quite effective. 
Analytical solutions for optimal control problems defined on low-dimensional quantum systems have been derived, leading to novel pulse sequences with unexpected gains compared with those traditionally used \cite{Khaneja01,D'Alessandro01,Boscain02,Stefanatos04,Bonnard09,Stefanatos10}, while numerical optimization methods, based on gradient algorithms or direct approaches, have been used to address more complex tasks and to minimize the effect of the ubiquitous experimental imperfections \cite{Peirce88,Khaneja05,Li09}.

At the heart of modern quantum technology lies the efficient cooling of trapped atoms, since it has created the ultimate physical systems thus far
for precision spectroscopy, frequency standards, and
even tests of fundamental physics \cite{Wieman99}, as well as candidate systems for quantum information processing \cite{Cirac04}.
In the present article we study a time-optimal control problem related to the frictionless cooling of atoms trapped in a time-dependent harmonic potential.
Frictionless atom cooling in a harmonic trapping potential is defined as the problem of
changing the harmonic frequency of the trap to some lower final value, while
keeping the populations of the initial and final levels invariant, thus
without generating friction and heating. Conventionally, an adiabatic process
is used where the frequency is changed slowly and the system follows the instantaneous
eigenvalues and eigenstates of the time-dependent Hamiltonian. The drawback of this
method is the long necessary times which may render it impractical. A way to bypass this problem is to use the theory of the
time-dependent quantum harmonic oscillator \cite{Lewis69} to prepare the same
final states and energies as with the adiabatic process at a given final time, without necessarily
following the instantaneous eigenstates at each moment. Achieving this goal in minimum time
has many important potential applications. For example, it can be used to
reach extremely low temperatures inaccessible by standard cooling techniques
\cite{Leanhardt03} and to reduce the velocity dispersion and collisional shifts
for spectroscopy and atomic clocks \cite{Bize05}. It is also closely related to the problem of
moving in minimum time a system between two thermal states \cite{Salamon09}.

It was initially
proved that minimum transfer time for the aforementioned problem can be achieved with ``bang-bang" real
frequency controls \cite{Salamon09}. Later, it was
shown that when the restriction for real frequencies is relaxed, allowing the
trap to become an expulsive parabolic potential at some time intervals,
shorter transfer times can be obtained, leading to a ``shortcut to adiabaticity" \cite{Chen10}.
In our recent work \cite{Stefanatos10PRA}, we formulated frictionless atom
cooling as a minimum-time optimal control problem, permitting the frequency to
take real and imaginary values in specified ranges. We showed that the optimal
solution has again a ``bang-bang" form and used this fact to obtain estimates of
the minimum transfer times for various numbers of switchings. In the present
article we complete our previous work by fully solving the corresponding
time-optimal control problem and obtaining the optimal synthesis.
The results presented here have already been used to determine the minimum time for transition between two thermal states \cite{Hoffmann_EPL11} and to show the emergence of the third law of classical thermodynamics from quantum thermodynamics \cite{Rezek11}, as highlighted in the conclusion.

\section{FORMULATION OF THE PROBLEM IN TERMS OF OPTIMAL CONTROL}

The evolution of the wavefunction $\psi(t,x)$ of a particle in a
one-dimensional parabolic trapping potential with time-varying frequency
$\omega(t)$ is given by the Schr\"{o}dinger equation
\begin{equation}
\label{Schrodinger}i\hbar\frac{\partial\psi}{\partial t}=\left[  -\frac
{\hbar^{2}}{2m}\frac{\partial^{2}}{\partial x^{2}}+\frac{m\omega^{2}%
(t)}{2}x^{2}\right]  \psi,
\end{equation}
where $m$ is the particle mass and $\hbar$ is Planck's constant; $x\in\mathbb{R}$ and $\psi$ is a square-integrable function on the real line. When
$\omega(t)$ is \emph{constant}, the above equation can be solved by separation
of variables and the solution is $\psi(t,x)=\sum_{n=0}^{\infty}c_{n} e^{-iE^{\omega
}_{n}t/\hbar}\Psi^{\omega}_{n}(x)$,
where $E^{\omega}_{n}$
are the eigenvalues and $\Psi^{\omega}_{n}(x)$
are the eigenfunctions of the quantum harmonic oscillator \cite{Merzbacher98}.
The coefficients $c_{n}$ can be found from
the initial condition $c_{n}=\int_{-\infty}^{\infty}\psi(0,x)\Psi^{\omega}
_{n}(x)dx$.

\begin{figure}[t]
\centering
\includegraphics[width=0.8\linewidth]{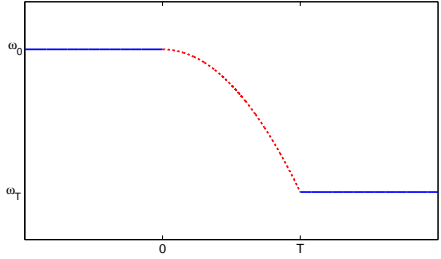}\caption{Time evolution of the
harmonic trap frequency.}%
\label{fig:frequency}%
\end{figure}

Consider now the case shown in Fig.\ \ref{fig:frequency}, where $\omega
(t)=\omega_{0}$ for $t\leq0$ and $\omega(t)=\omega_{T}<\omega_{0}$ for $t\geq
T$. This corresponds to a temperature reduction by a factor $\omega_{T}%
/\omega_{0}$, if the initial and final states are canonical \cite{Chen10}. For
frictionless cooling, the path $\omega(t)$ between these two values should be
chosen so that the populations of all the oscillator levels $n=0,1,2,\ldots$
for $t\geq T$ are equal to the ones at $t=0$. In other words, if
\begin{equation}
\label{initial_condition}\psi(0,x)=\sum_{n=0}^{\infty}c_{n}(0)\Psi^{\omega
_{0}}_{n}(x),\nonumber
\end{equation}
and
\begin{equation}
\label{final_condition}\psi(t,x)=\sum_{n=0}^{\infty}c_{n}(t)\Psi^{\omega_{T}%
}_{n}(x),\,t\geq T,\nonumber
\end{equation}
then frictionless cooling is achieved when
\begin{equation}
\label{frictionless_cooling}|c_{n}(t)|^{2}=|c_{n}(0)|^{2},\,t\geq T,
n=0,1,2,\ldots
\end{equation}
Among all the paths $\omega(t)$ that result in (\ref{frictionless_cooling}),
we would like to find one that achieves frictionless cooling in minimum time
$T$. In the following we provide a sufficient condition on $\omega(t)$ for
frictionless cooling and we use it to formulate the corresponding time-optimal
control problem.

\begin{proposition}
\label{prop:fr_cooling}\cite{Stefanatos11} If $\omega(t)$, with $\omega(0)=\omega_{0}$ and
$\omega(t)=\omega(T)=\omega_{T}$ for $t\geq T$ is such that the Ermakov
equation 
\begin{equation}
\label{Ermakov}\ddot{b}(t)+\omega^{2}(t)b(t)=\frac{\omega_{0}^{2}}{b^{3}(t)}%
\end{equation}
has a solution $b(t)$ with $b(0)=1,\dot{b}(0)=0$ and $b(t)=b(T)=(\omega
_{0}/\omega_{T})^{1/2},t\geq T$, then condition (\ref{frictionless_cooling})
for frictionless cooling is satisfied.
\end{proposition}

If we set
\begin{equation}
x_{1}=b,\,x_{2}=\frac{\dot{b}}{\omega_{0}},\,u(t)=\frac{\omega^{2}(t)}%
{\omega_{0}^{2}},
\end{equation}
and rescale time according to $t_{\mbox{new}}=\omega_{0} t_{\mbox{old}}$, we
obtain the following system of first order differential equations, equivalent
to the Ermakov equation 
\begin{align}
\label{system1}\dot{x}_{1}  &  = x_{2},\\
\label{system2}
\dot{x}_{2}  &  = -ux_{1}+\frac{1}{x_{1}^{3}}.
\end{align}
By incorporating the boundary conditions and possible restrictions on
$\omega(t)$ due, for example, to experimental limitations, and setting $\gamma=(\omega_{0}/\omega_{T})^{1/2}>1$, we obtain
the following time-optimal problem for frictionless cooling

\newtheorem{problem}{problem} \begin{problem}\label{problem}
Find $-u_1\leq u(t)\leq u_2$ with $u(0)=1, u(T)=1/\gamma^4$ such that starting from $(x_1(0),x_2(0))=(1,0)$, the above system reaches the final point $(x_1(T),x_2(T))=(\gamma,0), \gamma>1$, in minimum time $T$.
\end{problem}

The boundary conditions on the state variables $(x_{1},x_{2})$ are equivalent
to those for $b,\dot{b}$, while the boundary conditions on the control
variable $u$ are equivalent to those for $\omega$, so the requirements of Proposition \ref{prop:fr_cooling} are satisfied. Parameters $u_{1},u_{2}>0$
define the allowable values of $u(t)$ and it is $u_{2}\geq u(0)=1$. Note that
the possibility $\omega^{2}(t)<0$ (expulsive parabolic potential) for some
time intervals is permitted \cite{Chen10}. It is natural to
consider that also $u_{1}\geq1$, i.e. we can at least achieve the negative
potential $V(x)=-m\omega_{0}^{2}x^{2}/2$. Finally observe that the above
system describes the one-dimensional Newtonian motion of a unit-mass particle,
with position coordinate $x_{1}$ and velocity $x_{2}$. The acceleration
(force) acting on the particle is $-ux_{1}+1/x_{1}^{3}$. This point of view
can provide useful intuition about the time-optimal solution, as we will see later.


In the next section we solve the following optimal control problem

\begin{problem}\label{problem1}
Find $-u_1\leq u(t)\leq u_2$, with $u_1,u_2\geq 1$, such that starting from $(x_1(0),x_2(0))=(1,0)$, the system above reaches the final point $(x_1(T),x_2(T))=(\gamma,0), \gamma>1$, in minimum time $T$.
\end{problem}

In both problems the class of admissible controls formally are Lebesgue measurable functions that take values in the control set $[-u_1,u_2]$ almost everywhere. However, as we shall see, optimal controls are piecewise continuous, in fact bang-bang. The optimal control found for problem \ref{problem1} is also optimal for problem \ref{problem}, with the addition of instantaneous jumps at the initial and final points, so that the boundary conditions $u(0)=1$ and $u(T)=1/\gamma^4$ are satisfied. Note that in connection with Fig. \ref{fig:frequency}, a natural way to think about these conditions is that $u(t)=1$ for $t\leq 0$ and $u(t)=1/\gamma^4$ for $t\geq T$; in the interval $(0,T)$ we pick the control that achieves the desired transfer in minimum time.


\section{OPTIMAL SOLUTION}

The system described by (\ref{system1}), (\ref{system2}) can be expressed
in compact form as
\begin{equation}
\dot{x}=f(x)+ug(x), \label{affine}%
\end{equation}
where the vector fields are given by
\begin{equation}
f=\left(
\begin{array}
[c]{c}%
x_{2}\\
1/x_{1}^{3}%
\end{array}
\right)  ,\,\,g=\left(
\begin{array}
[c]{c}%
0\\
-x_{1}%
\end{array}
\right)
\end{equation}
and $x\in\mathcal{D}=\{(x_{1},x_{2})\in\mathbb{R}^{2}:x_{1}>0\}$ and $u\in
U=[-u_{1},u_{2}]$. Admissible controls are Lebesgue measurable functions that
take values in the control set $U$. Given an admissible control $u$ defined
over an interval $[0,T]$, the solution $x$ of the system (\ref{affine})
corresponding to the control $u$ is called the corresponding trajectory and we
call the pair $(x,u)$ a controlled trajectory. Note that the domain
$\mathcal{D}$\ is invariant in the sense that trajectories cannot leave
$\mathcal{D}$. Starting with any positive initial condition $x_{1}(0)>0$, and
using any admissible control $u$, as $x_{1}\rightarrow0^{+}$, the
\textquotedblleft repulsive force" $1/x_{1}^{3}$ leads to an increase in
$x_{1}$ that will keep $x_{1}$ positive (as long as the solutions exist).

For a constant $\lambda_{0}$ and a row vector $\lambda=(\lambda_{1}%
,\lambda_{2})\in\left(  \mathbb{R}^{2}\right)  ^{\ast}$ define the control
Hamiltonian as%
\[
H=H(\lambda_{0},\lambda,x,u)=\lambda_{0}+\langle\lambda,f(x)+ug(x)\rangle.
\]
Pontryagin's Maximum Principle \cite{Pontryagin}
provides the following necessary conditions for optimality:

\begin{theorem}
[\textrm{Maximum principle}%
]\cite{Pontryagin} \label{prop:max_principle} Let $(x_{\ast}(t),u_{\ast}(t))$
be a time-optimal controlled trajectory that transfers the initial condition
$x(0)=x_{0}$ into the terminal state $x(T)=x_{T}$. Then it is a necessary
condition for optimality that there exists a constant $\lambda_{0}\leq0$ and
nonzero, absolutely continuous row vector function $\lambda(t)$ such that:

\begin{enumerate}
\item $\lambda$ satisfies the so-called adjoint equation%
\[
\dot{\lambda}(t)=-\frac{\partial H}{\partial x}(\lambda_{0},\lambda
(t),x_{\ast}(t),u_{\ast}(t))
\]

\item For $0\leq t\leq T$ the function $u\mapsto H(\lambda_{0}%
,\lambda(t),x_{\ast}(t),u)$ attains its maximum\ over the control set $U$ at
$u=u_{\ast}(t)$.

\item $H(\lambda_{0},\lambda(t),x_{\ast}(t),u_{\ast}(t))\equiv0$.
\end{enumerate}
\end{theorem}

We call a controlled trajectory $(x,u)$ for which there exist multipliers
$\lambda_{0}$ and $\lambda(t)$ such that these conditions are satisfied an
extremal. Extremals for which $\lambda_{0}=0$ are called abnormal. If
$\lambda_{0}<0$, then without loss of generality we may rescale the $\lambda
$'s and set $\lambda_{0}=-1$. Such an extremal is called normal.

For the system (\ref{system1}), (\ref{system2}) we have
\begin{equation}
H(\lambda_{0},\lambda,x,u)=\lambda_{0}+\lambda_{1}x_{2}+\lambda_{2}\left(  \frac{1}%
{x_{1}^{3}}-x_{1}u\right)  ,\label{hamiltonian}%
\end{equation}
and thus
\begin{equation}
\dot{\lambda}=-\lambda\left(
\begin{array}
[c]{cc}%
0 & 1\\
-(u+3/x_{1}^{4}) & 0
\end{array}
\right)  =-\lambda A\label{adjoint}%
\end{equation}

Observe that $H$ is a linear function of the bounded control variable $u$. The
coefficient at $u$ in $H$ is $-\lambda_{2}x_{1}$ and, since $x_{1}>0$, its
sign is determined by $\Phi=-\lambda_{2}$, the so-called \emph{switching
function}. According to the maximum principle, point 2 above, the optimal
control is given by $u=-u_{1}$ if $\Phi<0$ and by $u=u_{2}$ if $\Phi>0$. The
maximum principle provides a priori no information about the control at times
$t$ when the switching function $\Phi$ vanishes. However, if $\Phi(t)=0$ and
$\dot{\Phi}(t)\neq0$, then at time $t$ the control switches between its
boundary values and we call this a bang-bang switch. If $\Phi$ were to vanish
identically over some open time interval $I$ the corresponding control is
called \emph{singular}.

\begin{proposition}
For Problem \ref{problem1} optimal controls are bang-bang.
\end{proposition}

\begin{proof}
Whenever the switching function $\Phi(t)=-\lambda_{2}(t)$ vanishes at some
time $t$, then it follows from the non-triviality of the multiplier
$\lambda(t)$ that its derivative $\dot{\Phi}(t)=-\dot{\lambda}_{2}%
(t)=\lambda_{1}(t)$ is non-zero. Hence the switching function changes sign and
there is a bang-bang switch at time $t$.
\end{proof}


\begin{definition}
We denote the vector fields corresponding to the constant bang controls
$-u_{1}$ and $u_{2}$ by $X=f-u_{1}g$ and $Y=f+u_{2}g$, respectively, and call
the trajectories corresponding to the constant controls $u\equiv-u_{1}$ and
$u\equiv u_{2}$ $X$- and $Y$-trajectories. A concatenation of an
$X$-trajectory followed by a $Y$-trajectory is denoted by $XY$ while the
concatenation in the inverse order is denoted by $YX$.
\end{definition}

In this paper we establish the precise concatenation sequences for optimal
controls and in particular calculate the times between switchings explicitly.

\begin{proposition}
All the extremals are normal.
\end{proposition}

\begin{proof}
If $(x,u)$ is an abnormal extremal trajectory that has a switching at time
$t$, then, since $\lambda_{2}(t)=0$, it follows from $H=0$ that we must have
$x_{2}(t)=0$. The starting point is $(1,0)$ and suppose that $u=-u_1$ initially.
From (\ref{system2}) it is $\dot{x}_{2}>0$ so $x_2>0$ and a switching at a point with $x_{2}(t)>0$, not allowed for an abnormal extremal, is
necessary in order to reach the target point $(\gamma,0)$.
If $u=u_2$ initially,
then $\dot{x}_{2}(0)=1-u_{2}<0$ and $x_2<0$ for some time interval. During this time it is
$\dot{x}_1<0$ and consequently $x_1<1<\gamma$. A switching is necessary,
which takes place on the $x_1$-axis for an abnormal extremal. The control changes to
$u=-u_1$ and the situation is as before, where one more switching is necessary at a point with $x_{2}(t)>0$,
forbidden for abnormal extremals. Thus, there are no abnormal extremals in the optimal solutions.
\end{proof}

For normal extremals we can set $\lambda_{0}=-1$.
Then, $H=0$ implies that for any switching time $t$ we must have $\lambda
_{1}(t)x_{2}(t)=1$. For an $XY$ junction we have $\dot{\Phi}%
(t)=\lambda_{1}(t)>0$ and thus necessarily $x_{2}(t)>0$ and analogously
optimal $YX$ junctions need to lie in $\{x_{2}<0\}$. We now develop the
precise structure of the switchings in a series of Lemmas.

\begin{lemma}
[\textrm{First integrals}] A first integral of the motion along the $X$-trajectory passing through $(\alpha,0)$ is
\begin{equation}
x_{2}^{2}-u_{1}x_{1}^{2}+\frac{1}{x_{1}^{2}}=-u_{1}\alpha^{2}+\frac{1}{\alpha^{2}},\label{first_integral_1}
\end{equation}
\end{lemma}
while a first integral of the motion along the $Y$-trajectory passing through $(\beta,0)$ is
\begin{equation}
x_{2}^{2}+u_{2}x_{1}^{2}+\frac{1}{x_{1}^{2}}=u_{2}\beta^{2}+\frac{1}{\beta^{2}}.\label{first_integral_2}
\end{equation}

\begin{proof}
Use the system equations (\ref{system1}) and (\ref{system2}).
\end{proof}

\begin{lemma}
[\textrm{Inter-switching time}]\label{switch_time} Let $p=(x_{1},x_{2})$ be a
switching point and $\tau$ denote the time to reach the next switching point
$q$. If $\overrightarrow{pq}$ is a $Y$-trajectory, then
\begin{equation}
\sin(2\sqrt{u_{2}}\tau)=-\frac{2\sqrt{u_{2}}x_{1}x_{2}}{x_{2}^{2}+u_{2}%
x_{1}^{2}},\quad\cos(2\sqrt{u_{2}}\tau)=\frac{x_{2}^{2}-u_{2}x_{1}^{2}}%
{x_{2}^{2}+u_{2}x_{1}^{2}}\label{par_2}%
\end{equation}
while, if $\overrightarrow{pq}$ is an $X$-trajectory, then
\begin{equation}
\sinh(2\sqrt{u_{1}}\tau)=-\frac{2\sqrt{u_{1}}x_{1}x_{2}}{x_{2}^{2}-u_{1}%
x_{1}^{2}},\quad\cosh(2\sqrt{u_{1}}\tau)=\frac{x_{2}^{2}+u_{1}x_{1}^{2}}%
{x_{2}^{2}-u_{1}x_{1}^{2}}.\label{par_1}%
\end{equation}

\end{lemma}

Note that the inter-switching times depend only on the ratio $x_2/x_1$.

\begin{proof}
These formulas are obtained as an application of the concept of a
``conjugate point" for bang-bang controls \cite{Su1}.
Without loss of generality assume that the trajectory passes through $p$ at
time $0$ and is at $q$ at time $\tau$. Since $p$ and $q$ are switching points,
the corresponding multipliers vanish against the control vector field $g$ at
those points, i.e., $\langle\lambda(0),g(p)\rangle=\langle\lambda
(\tau),g(q)\rangle=0$. We need to compute what the relation $\langle
\lambda(\tau),g(q)\rangle=0$ implies at time $0$. In order to do so, we move
the vector $g(q)$ along the $Y$-trajectory backward from $q$ to $p$. This is
done by means of the solution $w(t)$ of the variational equation along the
$Y$-trajectory with terminal condition $w(\tau)=g(q)$ at time $\tau$. Recall
that the variational equation along $Y$ is the linear system $\dot{w}=Aw$
where $A$ is given in (\ref{adjoint}). Symbolically, if we denote by
$e^{tY}(p)$ the value of the $Y$-trajectory at time $t$ that starts at the
point $p$ at time $0$ and by $(e^{-tY})_{\ast}$ the backward evolution under
the linear differential equation $\dot{w}=Aw$, then we can represent this
solution in the form
\begin{align}
w(0)&=(e^{-\tau Y})_{\ast}w(\tau)=(e^{-\tau Y})_{\ast}g(q)\nonumber\\&=(e^{-\tau Y})_{\ast}g(e^{\tau
Y}(p))=(e^{-\tau Y})_{\ast}\circ
g\circ e^{\tau Y}(p).\nonumber
\end{align}
Since the
\textquotedblleft adjoint equation\textquotedblright\ of the Maximum Principle
is precisely the adjoint equation to the variational equation, it follows that
the function $t\mapsto\langle\lambda(t),w(t)\rangle$ is constant along the
$Y$-trajectory. Hence $\langle\lambda(\tau),g(q)\rangle=0$ implies that
\[
\langle\lambda(0),w(0)\rangle=\langle\lambda(0),(e^{-\tau Y})_{\ast}g(e^{\tau
Y}(p))\rangle=0
\]
as well. But the non-zero multiplier $\lambda(0)$ can only be orthogonal to
both $g(p)$ and $w(0)$ if these vectors are parallel, $g(p)\Vert
w(0)=(e^{-\tau Y})_{\ast}g(e^{\tau Y}(p))$. It is this relation that defines
the switching time.

It remains to compute $w(0)$. For this we make use of the well-known relation \cite{Jurdjevic97}
\begin{equation}
(e^{-\tau Y})_{\ast}\circ g\circ e^{\tau Y}=e^{\tau\,adY}(g)\label{adj}%
\end{equation}
where the operator $adY$ is defined as $adY(g)=[Y,g]$, with $[,]$ denoting the
Lie bracket of the vector fields $Y$ and $g$. 
For our system, the Lie algebra $\mathcal{L}$
generated by the fields $f$ and $g$ actually is finite dimensional: we have
\[
\lbrack f,g](x)=\left(
\begin{array}
[c]{c}%
x_{1}\\
-x_{2}%
\end{array}
\right)
\]
and the relations
\[
\lbrack f,[f,g]]=2f,\qquad\lbrack g,[f,g]]=-2g
\]
can be directly verified. Using these relations and the analyticity of the
system, $e^{t\,adY}(g)$ can be calculated in closed form from the expansion
\begin{equation}
e^{t\,adY}(g)=\sum_{n=0}^{\infty}\frac{t^{n}}{n!}\,ad\,^{n}Y(g),
\end{equation}
where, inductively, $ad^{n}Y(g)=[Y,ad^{n-1}Y(g)]$.
It is not hard to show that for $n=0,1,2,\ldots$, we have that
\[
ad\,^{2n+1}Y(g)=(-4u_{2})^{n}[f,g]
\]
and
\[
ad\,^{2n+2}Y(g)=2(-4u_{2})^{n}(f-u_{2}g),
\]
so that
\begin{align}
e^{t\,adY}(g)&=g+\sum_{n=0}^{\infty}\frac{t^{2n+1}}{(2n+1)!}\,(-4u_{2}%
)^{n}[f,g]\nonumber\\&+\sum_{n=0}^{\infty}\frac{2t^{2n+2}}{(2n+2)!}\,(-4u_{2})^{n}%
(f-u_{2}g).\nonumber
\end{align}
By summing the series appropriately we obtain
\begin{align}
e^{t\,adY}(g)&=g+\frac{1}{2\sqrt{u_{2}}}\sin(2\sqrt{u_{2}}t)[f,g]\nonumber\\&+\frac
{1}{2u_{2}}[1-\cos(2\sqrt{u_{2}}t)](f-u_{2}g).\nonumber
\end{align}
Hence the field $w(0)=(e^{-\tau Y})_{\ast}g(e^{\tau Y}(p))$ is
parallel to $g(p)=(0,-x_{1})^{T}$ if and only if
\begin{equation}
\sqrt{u_{2}}x_{1}\sin(2\sqrt{u_{2}}\tau)+x_{2}\left[  1-\cos(2\sqrt{u_{2}}%
\tau)\right]  =0.\nonumber
\end{equation}
Hence
\begin{equation}
\sin(2\sqrt{u_{2}}\tau)=-\frac{x_{2}}{\sqrt{u_{2}}x_{1}}[1-\cos(2\sqrt{u_{2}%
}\tau)]
\end{equation}
from which (\ref{par_2}) follows. Note that the solution $\cos(2\sqrt{u_{2}%
}\tau)=1$ is rejected because it corresponds to $\tau=0$ or $\tau=\pi/\sqrt{u_{2}}$,
the latter being the period of the closed trajectory.

Working analogously for an $X$-trajectory  we obtain (\ref{par_1}).
\end{proof}

\begin{figure}[t]
\centering
\includegraphics[width=0.8\linewidth]{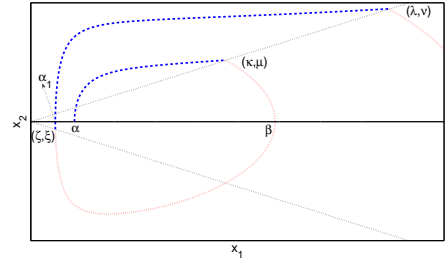}\caption{Consecutive switching points lie on two opposite-slope lines through the origin. Blue dashed curves correspond to $X$-segments, red dotted curves to $Y$-segments.}%
\label{fig:threeswitchings}%
\end{figure}

\begin{lemma}
[\textrm{Main technical point}]\label{prop:ratio} The ratio of the coordinates
of consecutive switching points has constant magnitude but alternating sign,
while these points are not symmetric with respect to the $x_{1}$-axis.
\end{lemma}

\begin{proof}
Consider the trajectory shown in Fig.\ \ref{fig:threeswitchings}, with
switching points $(\kappa,\mu), (\zeta,\xi)$ and $(\lambda,\nu)$.
Starting from $(\kappa,\mu)$ and integrating the equations of motion (\ref{system1}) and (\ref{system2}) for the inter-switching time given in (\ref{par_2}), we can find the coordinates of the next switching point and then show that $\xi/\zeta=-\mu/\kappa$ while $(\zeta,\xi)\neq (\kappa,-\mu)$ \cite{Stefanatos11}. Subsequently, integrating the equations for the inter-switching time given in (\ref{par_1}), we can also show that $\nu/\lambda=-\xi/\zeta$ and $(\lambda,\nu)\neq (\zeta,-\xi)$ \cite{Stefanatos11}.

Alternatively, we present a more elegant proof based on the symmetries of the system. Observe that the transformation $(t,x_1,x_2)\rightarrow (-t,x_1,-x_2)$ leaves the system (\ref{system1}) and (\ref{system2}) invariant for constant $u$. So, starting from $(\zeta,\xi)$ and running the system backwards to the next (previous in the forward direction) switching point $(\kappa,\mu)$, the switching time is given by a relation similar to (\ref{par_2}), with $x_2/x_1=-\xi/\zeta$. But this switching time is the same as in the forward direction, where $x_2/x_1=\mu/\kappa$ in (\ref{par_2}). Then, using (\ref{par_2}), it is not hard to see that $\xi/\zeta=-\mu/\kappa$. Note that
$\zeta=\kappa$ would imply $\xi=-\mu$, i.e. returning to the same point on the $x_1$-axis with opposite velocity before switching again, which is obviously not time optimal. Thus, it is $\zeta\neq\kappa$ in general so $(\zeta,\xi)\neq (\kappa,-\mu)$.

\end{proof}

\begin{figure}[t]
\centering
\includegraphics[width=0.8\linewidth]{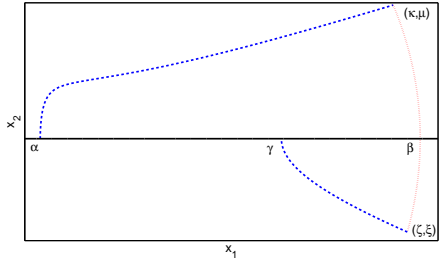}\caption{Blue dashed curves correspond to $X$-segments, red dotted curves to $Y$-segments.}%
\label{fig:forbidden}%
\end{figure}



In the following proposition we use Lemma \ref{prop:ratio} to determine the
form of the optimal trajectory.

\begin{proposition}
[\textrm{Form of the optimal trajectory}]\label{prop:spiral} The optimal
trajectory can have the one-switching form $XY$ or the spiral form
$YX\ldots YXY$ with an even number of switchings.
\end{proposition}

\begin{proof}
We first show that when the optimal trajectory has more than one switching,
it cannot start with an $X$-segment. For just two switchings, consider the
trajectory $XYX$ depicted in Fig.\ \ref{fig:forbidden}, where $\alpha=1$
(starting point), $(\gamma,0), \gamma>1$ is the target point and $(\kappa
,\mu),(\zeta,\xi)$ are the switching points. Since both of the switching
points belong to the $Y$-segment passing through $(\beta,0)$, their
coordinates satisfy (\ref{first_integral_2}). If we denote by $s$ the common
ratio $\mu^{2}/\kappa^{2}=\xi^{2}/\zeta^{2}=s$,
then both $\kappa,\zeta$ satisfy the equation
\begin{equation}
(s+u_{2})x_{1}^{4}-(u_{2}\beta^{2}+\frac{1}{\beta^{2}})x_{1}^{2}+1=0,\nonumber
\end{equation}
so
\begin{equation}
\kappa^{2}\zeta^{2}=\frac{1}{s+u_{2}}<1,\nonumber
\end{equation}
since $u_{2}\geq1,s>0$. But also $\kappa^{2}\zeta^{2}>1$,
since $\kappa^{2}>1$ and $\zeta^{2}>\gamma^{2}>1$. Thus this trajectory cannot
be optimal.

For more switchings, consider the case shown in Fig
\ref{fig:threeswitchings}, where now $\alpha=1$, and use $s$ to denote the
common ratio of the squares of the coordinates at the switching points. If $\tau$
is the switching time between $(\zeta,\xi)$ and $(\lambda,\nu)$, then from
(\ref{par_1}) we obtain
\begin{equation}
\frac{s}{u_{1}}=\frac{\cosh(2\tau\sqrt{u_{1}})+1}{\cosh(2\tau\sqrt{u_{1}})-1}>1.\nonumber
\end{equation}
But $s/u_1=\mu^2/(u_1\kappa^2)$, and from (\ref{first_integral_1}) we find ($\alpha=1$)
\begin{equation}
\frac{s}{u_{1}}=\frac{(u_{1}\kappa^{2}+1)(\kappa^{2}-1)}{u_{1}\kappa^{4}%
}<1\Leftrightarrow(u_{1}-1)\kappa^{2}>-1,\nonumber
\end{equation}
since $u_{1}\geq1$. Thus if the optimal trajectory has more than one
switching, it needs to start with a $Y$-segment.

We next show that the optimal trajectory reaches the target point $(\gamma,0),
\gamma>1$ with a $Y$-segment. This is obviously the case for one switching,
and also for two switchings since only the $YXY$ trajectory is permitted
(the $XYX$ was excluded above). For more than two switchings consider the
situation shown in Fig.\ \ref{fig:forbidden}. It is $\mu^{2}/\kappa^{2}=\xi
^{2}/\zeta^{2}=s$ and $s>u_{1}$ since at least one $YXY$-segment is included
in the trajectory. Point $(\zeta,\xi)$ belongs to the final $X$-segment ending
to $(\gamma,0)$, so
\begin{equation}
(s-u_{1})\zeta^{2}+\frac{1}{\zeta^{2}}=-u_{1}\gamma^{2}+\frac{1}{\gamma^{2}}.\nonumber
\end{equation}
The left hand side is positive, since $s>u_{1}$, while the right had side is
negative, since $\gamma>1,u_{1}\geq1$. Thus the optimal trajectory reaches the
target point with a $Y$-segment.
\end{proof}

\begin{corollary}
\label{uone}
For $|u|\leq1$ the optimal solution has only one switching.
\end{corollary}

\begin{proof}
For $u=u_{2}=1$ the starting point $(1,0)$ is an equilibrium point of system (\ref{system1}), (\ref{system2}). So the optimal trajectory cannot start with
a $Y$-segment. The only trajectory thus permitted is $XY$
\end{proof}

From Proposition \ref{prop:spiral} we see that the optimal trajectory can have
aside from the expected one-switching form, shown in Fig.\ \ref{fig:upper}, the
spiral form shown in Fig.\ \ref{fig:spiral}. An intuitive understanding of this
latter form can be obtained by viewing system equations (\ref{system1}), (\ref{system2}) as describing the motion of a unit mass particle with position
$x_{1}$ and velocity $x_{2}$. In light of this interpretation we see that
along a spiral trajectory the particle, instead of moving directly to the
target, goes close to $x_{1}=0$ where there is a strong repulsive potential
($1/x_{1}^{3}$) to acquire speed and reach the target point faster. In the
following theorem we calculate the transfer time for the candidate optimal trajectories.


\begin{figure}[t]
\centering
\begin{tabular}
[c]{c}%
\subfigure[$\ $]{ \label{fig:upper} \includegraphics[width=.8\linewidth]{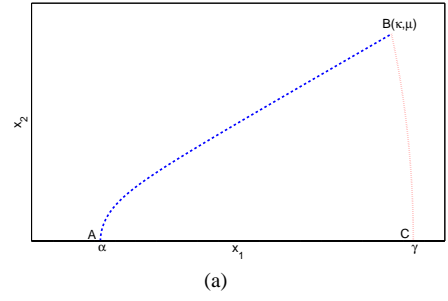}}\\
\subfigure[$\ $]{ \label{fig:spiral} \includegraphics[width=.8\linewidth]{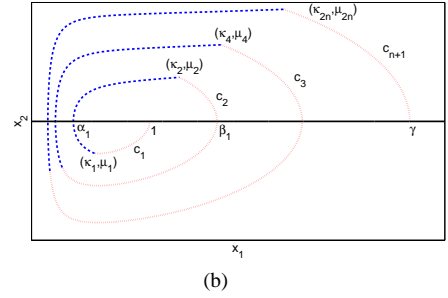}}
\end{tabular}
\caption{(a) Trajectory with one switching (zero turns) (b) Trajectory with
$n$ turns.}%
\label{fig:opt_traj_final}%
\end{figure}

\begin{theorem}
\label{prop:time} Starting from $(1,0)$, the necessary time to reach the
target point $(\gamma,0), \gamma>1$ with one switching is
\begin{align}
\label{time_0}T_{0}&=\frac{1}{\sqrt{u_{1}}}\sinh^{-1}\left(  \sqrt{\frac
{u_{1}(\gamma^{2}-1)(u_{2}\gamma^{2}-1)}{\gamma^{2}(u_{1}+u_{2})(u_{1}+1)}%
}\right)\nonumber\\ & +\frac{1}{\sqrt{u_{2}}}\sin^{-1}\left(  \sqrt{\frac{u_{2}(\gamma
^{2}-1)(u_{1}\gamma^{2}+1)}{(u_{1}+u_{2})(u_{2}\gamma^{4}-1)}}\right)  .
\end{align}
The necessary time to reach the target with $n$ turns ($2n$ switchings) is
\begin{equation}
\label{time_n}T_{n}=T_{I}+nT_{X}+(n-1)T_{Y}+T_{F},
\end{equation}
where
\begin{align}
\label{time_in}T_{I}  &  = \frac{1}{2\sqrt{u_{2}}}\cos^{-1}\left(
-\frac{sc_{1}+u_{2}\sqrt{c_{1}^{2}-4(s+u_{2})}}{(s+u_{2})\sqrt{c_{1}%
^{2}-4u_{2}}}\right),\\
\label{time_fi}
T_{F}  &  = \frac{1}{2\sqrt{u_{2}}}\cos^{-1}\left(  \frac{-sc_{n+1}+u_{2}%
\sqrt{c_{n+1}^{2}-4(s+u_{2})}}{(s+u_{2})\sqrt{c_{n+1}^{2}-4u_{2}}}\right)  ,
\end{align}
\begin{align}
\label{switch_X}T_{X}  &  = \frac{1}{2\sqrt{u_{1}}}\cosh^{-1}\left(
\frac{s+u_{1}}{s-u_{1}}\right),\\
\label{switch_Y}
T_{Y}  &  = \frac{1}{2\sqrt{u_{2}}}\left(  2\pi-\cos^{-1}\left(  \frac
{s-u_{2}}{s+u_{2}}\right)  \right)  ,
\end{align}
\begin{align}
\label{c_1}c_{1}  &  =u_{2}+1,\\
\label{c_n}
c_{n+1}  &  =u_{2}\gamma^{2}+\frac{1}{\gamma^{2}},
\end{align}
and $s$ is the solution of the transcendental equation
\begin{equation}
\label{transcendental}\frac{c_{1}+\sqrt{c_{1}^{2}-4(s+u_{2})}}{c_{n+1}%
+\sqrt{c_{n+1}^{2}-4(s+u_{2})}}=\left(  \frac{s-u_{1}}{s+u_{2}}\right)  ^{n}%
\end{equation}
in the interval $u_{1}<s\leq(u_{2}-1)^{2}/4$. The constants $c_{1}$ and
$c_{n+1}$ characterize the first and the last $Y$-segments, respectively, of
the trajectory. The number of turns satisfies the following inequality
\begin{equation}
\label{turns}
n\leq\left[\frac{T_0}{T_X(s_+)}\right],
\end{equation}
where $s_+=(u_{2}-1)^{2}/4$ and $[\,]$ denotes the integer part.
\end{theorem}

\begin{proof}
In Fig.\ \ref{fig:upper} we show a trajectory with one switching point
$B(\kappa,\mu)$. The coordinates of this point satisfy equations (\ref{first_integral_1}) and (\ref{first_integral_2}) with $\alpha=1$ and $\beta=\gamma$, 
\begin{eqnarray}
\mu_2^2-u_1\kappa^2+\frac{1}{\kappa^2}&=&1-u_1,\nonumber\\
\mu^2+u_2\kappa^2+\frac{1}{\kappa^2}&=&u_2\gamma^2+\frac{1}{\gamma^2},\nonumber
\end{eqnarray}
from which we find
\begin{equation}
\kappa^{2}=\frac{u_{2}\gamma^{4}+1+\gamma^{2}(u_{1}-1)}{\gamma^{2}(u_{1}%
+u_{2})}.\nonumber
\end{equation}
By integrating the equations of motion (\ref{system1}) and (\ref{system2}),
we find the necessary time along each segment of the trajectory, $AB$ and $BC$.
The total transfer time is the sum of these times and is given by (\ref{time_0}). Next
consider the case with $n$ turns and $2n$ switching points $(\kappa_{j}%
,\mu_{j})$, Fig.\ \ref{fig:spiral}, with constant ratio $\mu_{j}^{2}/\kappa
_{j}^{2}=s$. The first switching point satisfies the equations
\begin{align}
\label{curv_1_b}\mu_{1}^{2}+u_{2}\kappa_{1}^{2}+\frac{1}{\kappa_{1}^{2}}  &
=c_{1},\\
\label{curv_1_a}
\mu_{1}^{2}-u_{1}\kappa_{1}^{2}+\frac{1}{\kappa_{1}^{2}}  &  =c,
\end{align}
where $c_{1}$ is given by (\ref{c_1}) and $c=-u_{1}\alpha_{1}^{2}+1/\alpha
_{1}^{2}$, while the second switching point satisfies
\begin{align}
\label{curv_2_b}\mu_{2}^{2}+u_{2}\kappa_{2}^{2}+\frac{1}{\kappa_{2}^{2}}  &
=c_{2},\\
\label{curv_2_a}
\mu_{2}^{2}-u_{1}\kappa_{2}^{2}+\frac{1}{\kappa_{2}^{2}}  &  =c,
\end{align}
where $c_{2}=u_{2}\beta_{1}^{2}+1/\beta_{1}^{2}$. The constants $c_{1}$ and $c_{2}$
characterize the first and second $Y$-segments of the trajectory, while
the constant $c$ characterizes the $X$-segment joining them. Subtracting
(\ref{curv_1_a}) from (\ref{curv_2_a}) and using Lemma \ref{prop:ratio} which
assures that $\kappa_{1}\neq\kappa_{2}$ (consecutive switching points are not
symmetric with respect to $x_{1}$-axis) we find that
\begin{equation}
\label{s_equation}s-u_{1}-\frac{1}{\kappa_{1}^{2}\kappa_{2}^{2}}=0.
\end{equation}
But from (\ref{curv_1_b}), (\ref{curv_2_b}) and the constant ratio relation we
find
\begin{align}
\kappa_{1}^{2}  &  =\frac{2}{c_{1}+\sqrt{c_{1}^{2}-4(s+u_{2})}},\nonumber\\
\kappa_{2}^{2}  &  =\frac{c_{2}+\sqrt{c_{2}^{2}-4(s+u_{2})}}{2(s+u_{2})},\nonumber
\end{align}
where, while solving the quadratic equations we used the $-$ sign for the
first and the $+$ sign for the second switching point. The choice of sign for
the first switching point will be justified below, while the choice of sign
for consecutive switching points should be alternating to avoid picking the
symmetric image of the previous point. Using these relations,
(\ref{s_equation}) takes the form
\begin{equation}
\label{transcendental2}\frac{c_{1}+\sqrt{c_{1}^{2}-4(s+u_{2})}}{c_{2}%
+\sqrt{c_{2}^{2}-4(s+u_{2})}}=\frac{s-u_{1}}{s+u_{2}}.\nonumber
\end{equation}
By repeating the above procedure for all the consecutive pairs of switching
points, we find
\begin{equation}
\label{transcendental3}\frac{c_{i}+\sqrt{c_{i}^{2}-4(s+u_{2})}}{c_{i+1}%
+\sqrt{c_{i+1}^{2}-4(s+u_{2})}}=\frac{s-u_{1}}{s+u_{2}},\,i=1,2,\ldots,n.\nonumber
\end{equation}
Multiplying the above equations we obtain (\ref{transcendental}), \emph{one}
transcendental equation for the ratio $s$. If we choose the $+$ sign in the
quadratic equation for $\kappa_{1}^{2}$, we obtain an equation similar to
(\ref{transcendental}) but with inverted left hand side. It is $c_{n+1}%
>c_{1}\Leftrightarrow(\gamma^{2}-1)(u_{2}\gamma^{2}-1)>0$ and $c_1,c_{n+1}>0$, so
\begin{equation}
\frac{c_{n+1}+\sqrt{c_{n+1}^{2}-4(s+u_{2})}}{c_{1}+\sqrt{c_{1}^{2}-4(s+u_{2}%
)}}>1>\left(  \frac{s-u_{1}}{s+u_{2}}\right)  ^{n},\nonumber
\end{equation}
and the corresponding transcendental equation has no solution. Note that the
left hand side of (\ref{transcendental}) is a decreasing function of $s$ while
the right hand side is an increasing one, so if a solution exists, it is
unique. The ratio is bounded below by the requirement $s/u_{1}>1$ and above by
$c_{1}^{2}-4(s+u_{2})\geq0\Leftrightarrow s\leq s_+=(u_{2}-1)^{2}/4$. This is also
the maximum value of $s$ on the first $Y$-segment (\ref{curv_1_b}). Once we
have calculated this ratio, we can find the time interval between consecutive
switchings using (\ref{switch_X}) for an $X$-segment and (\ref{switch_Y}) for
a $Y$-segment, relations obtained from Lemma \ref{switch_time} on the inter-switching time. Observe that the times along all intermediate $X$- (respectively $Y$-) trajectories are equal. The initial time interval $T_{I}$ (from the starting point up
to the first switching) and the final time interval $T_{F}$ (from the last
switching up to the target point) can be easily calculated and are given in (\ref{time_in}) and (\ref{time_fi}), respectively. The total duration $T_{n}$ of
the trajectory with $n$ turns joining the points $(1,0)$ and $(\gamma,0)$ is
given by (\ref{time_n}). Observe that $T_n(s)> nT_X(s)\geq nT_X(s_+)$, where the last inequality follows from the fact that $T_X$ is a decreasing function of $s$, see (\ref{switch_X}). A solution with $n$ turns can be candidate for optimality only if the number of turns is bounded as in (\ref{turns}). Otherwise we have $T_n(s)>T_0$ and the one-switching strategy is faster.
\end{proof}


Using Theorem \ref{prop:time} we can find the times $T_{n}$ for a specific
target $(\gamma, 0)$ and compare them to obtain the minimum time. 

\begin{figure}[t]
\centering
\begin{tabular}
[c]{c}%
\subfigure[$\ $]{ \label{fig:times2} \includegraphics[width=.8\linewidth]{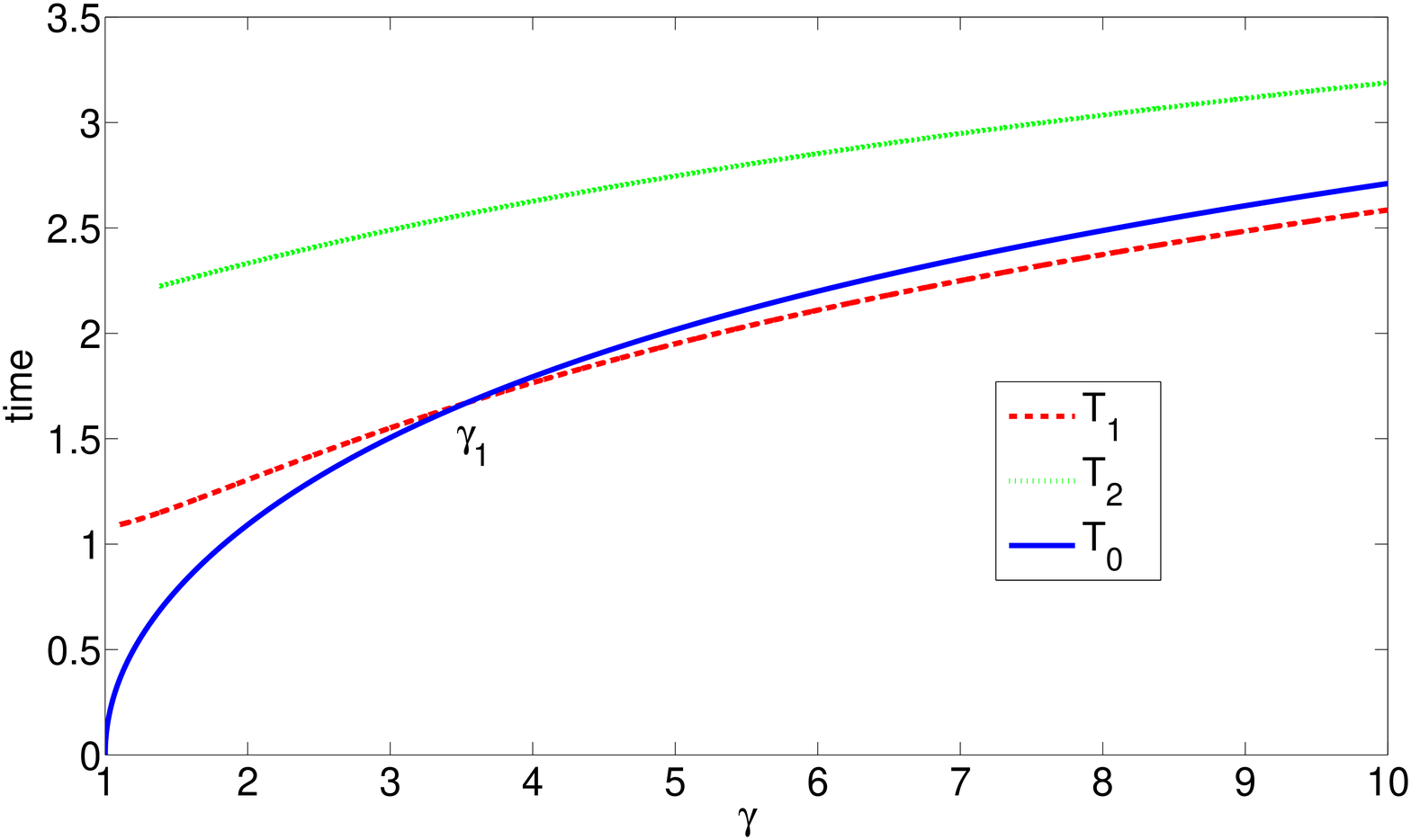}}\\
\subfigure[$\ $]{ \label{fig:synthesis2} \includegraphics[width=.8\linewidth]{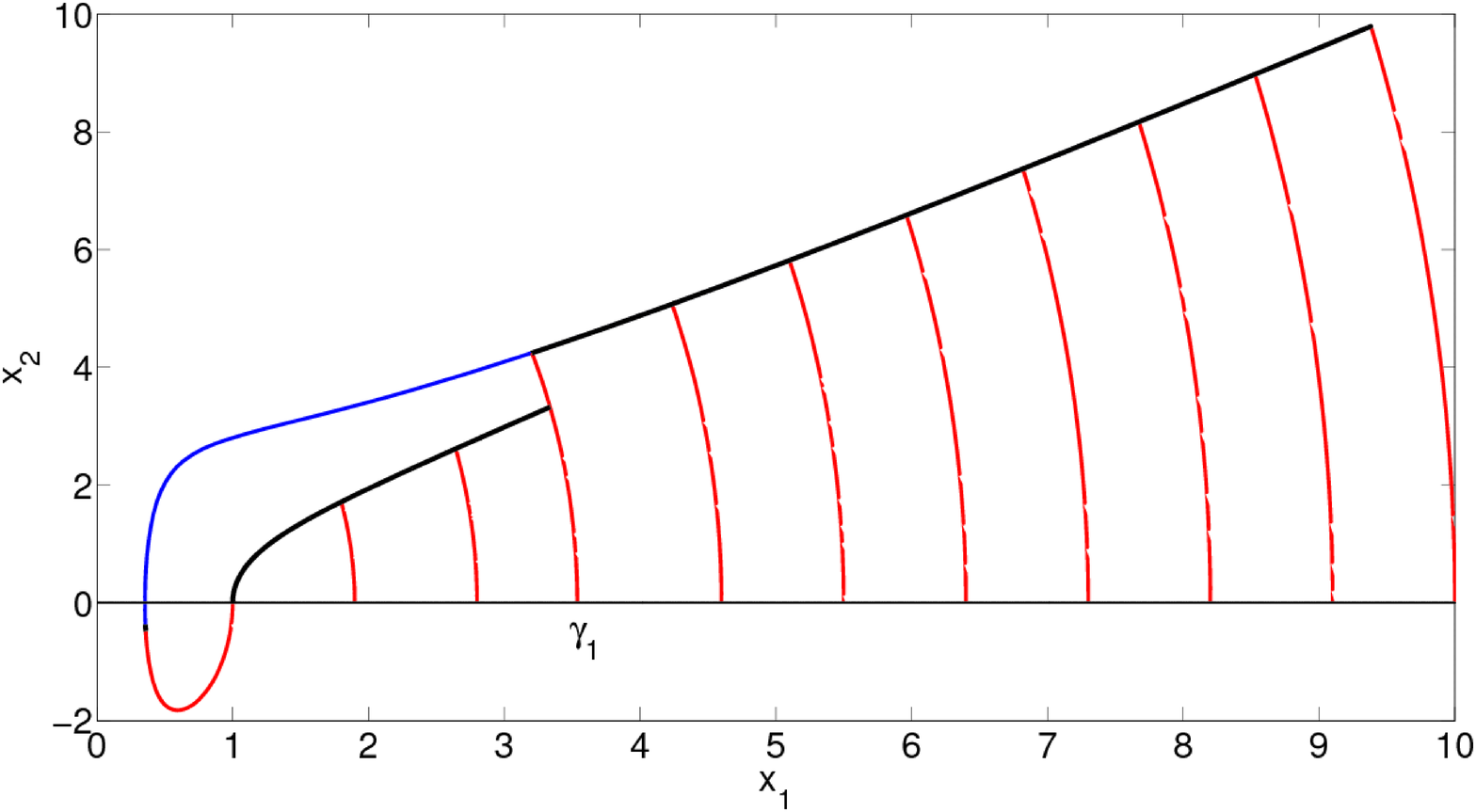}}
\end{tabular}
\caption{(a) Transfer times corresponding to zero, one and two turns for
$u_{1}=1, u_{2}=8, \gamma\in[1,10]$. (b) Switching curves (black solid curves) and
characteristic optimal trajectories starting from $(1,0)$.}%
\label{fig:twomax}%
\end{figure}

\section{EXAMPLES}

In Fig.\ \ref{fig:times2} we plot the times $T_{0}, T_{1}$ and $T_{2}$ from Theorem \ref{prop:time}, corresponding to zero, one and two turns, for
$u_{1}=1, u_{2}=8$ and $\gamma\in[1,10]$. For $\gamma\leq\gamma_{1}$ the
strategy with zero turns (one switching) is optimal, while for $\gamma
\geq\gamma_{1}$ it is the strategy with one turn (up to the range of $\gamma$
plotted). The point $(\gamma_{1},0)$ can be reached with both strategies in
equal time, that is, it belongs to the cut-locus of these two control sequences from $(1,0)$. Note that the strategies
with one and two turns are feasible after some $\gamma>1$, where the
transcendental equation (\ref{transcendental}) has a solution. In Fig.\
\ref{fig:synthesis2} we plot the switching curves (black solid curves) as well as
some characteristic optimal trajectories starting from $(1,0)$. For
$\gamma\leq\gamma_{1}$ the optimal trajectory starts with an $X$-segment that
coincides with the switching curve (black solid curve) passing from $(1,0)$. It
switches at some point and then travels along a $Y$-segment (red dotted curve) to
meet the $x_{1}$-axis. For $\gamma\geq\gamma_{1}$ the optimal trajectory
starts with a $Y$-segment (red dotted curve passing from $(1,0)$) and switches at
some point in the tiny black area of this curve to an $X$-segment (blue dashed curve).
Then it meets at some point the second switching curve on the upper quadrant
and changes to a $Y$-segment (red dotted curve) that hits the $x_{1}$-axis at the target
point. Note that the optimal trajectories between the two switchings (blue dashed
curves) are very close to the second switching curve on the upper quadrant and
they are not shown entirely.

\begin{figure}[t]
\centering
\begin{tabular}
[c]{c}%
\subfigure[$\ $]{ \label{fig:times4} \includegraphics[width=.8\linewidth]{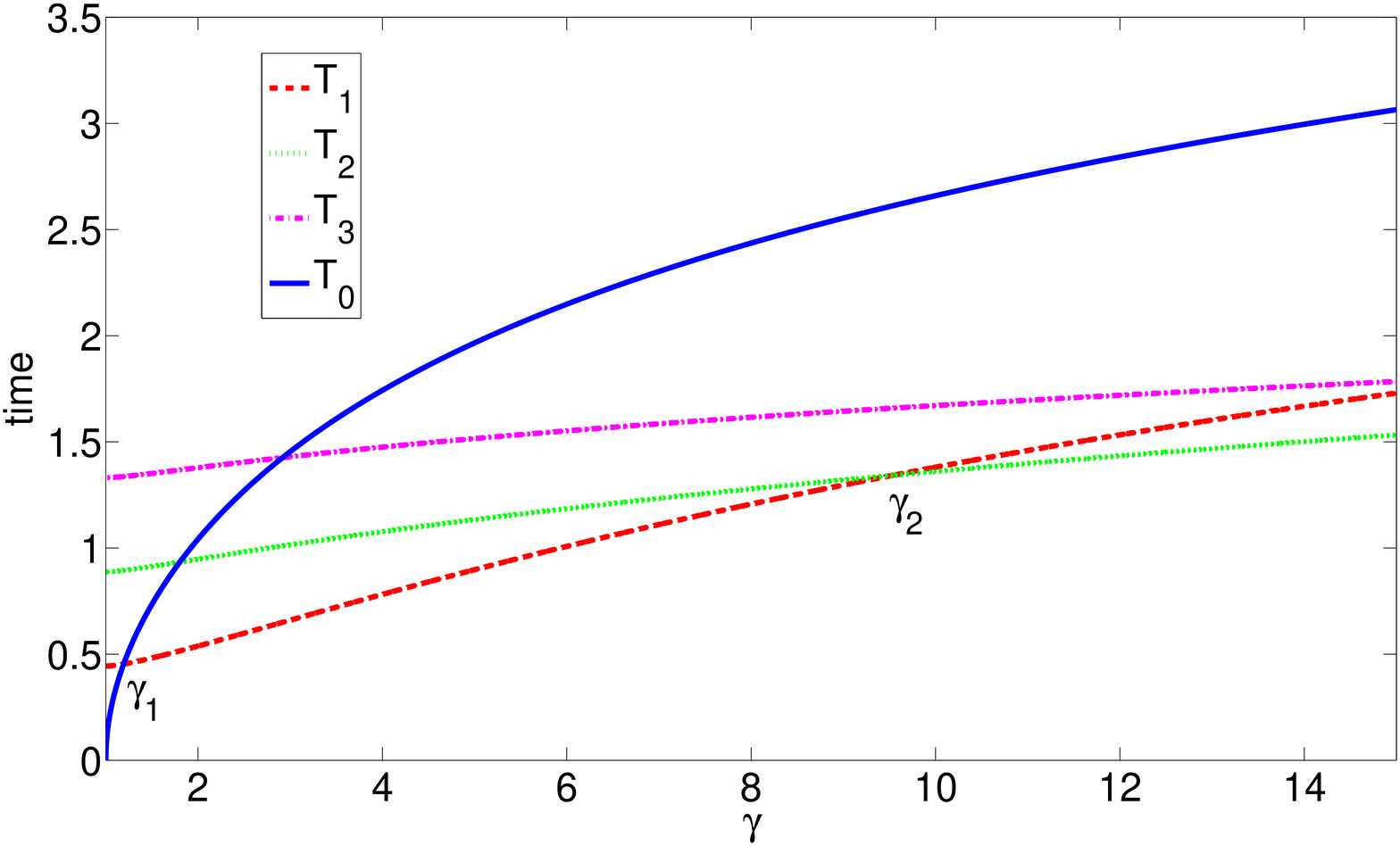}}\\
\subfigure[$\ $]{ \label{fig:synthesis4} \includegraphics[width=.8\linewidth]{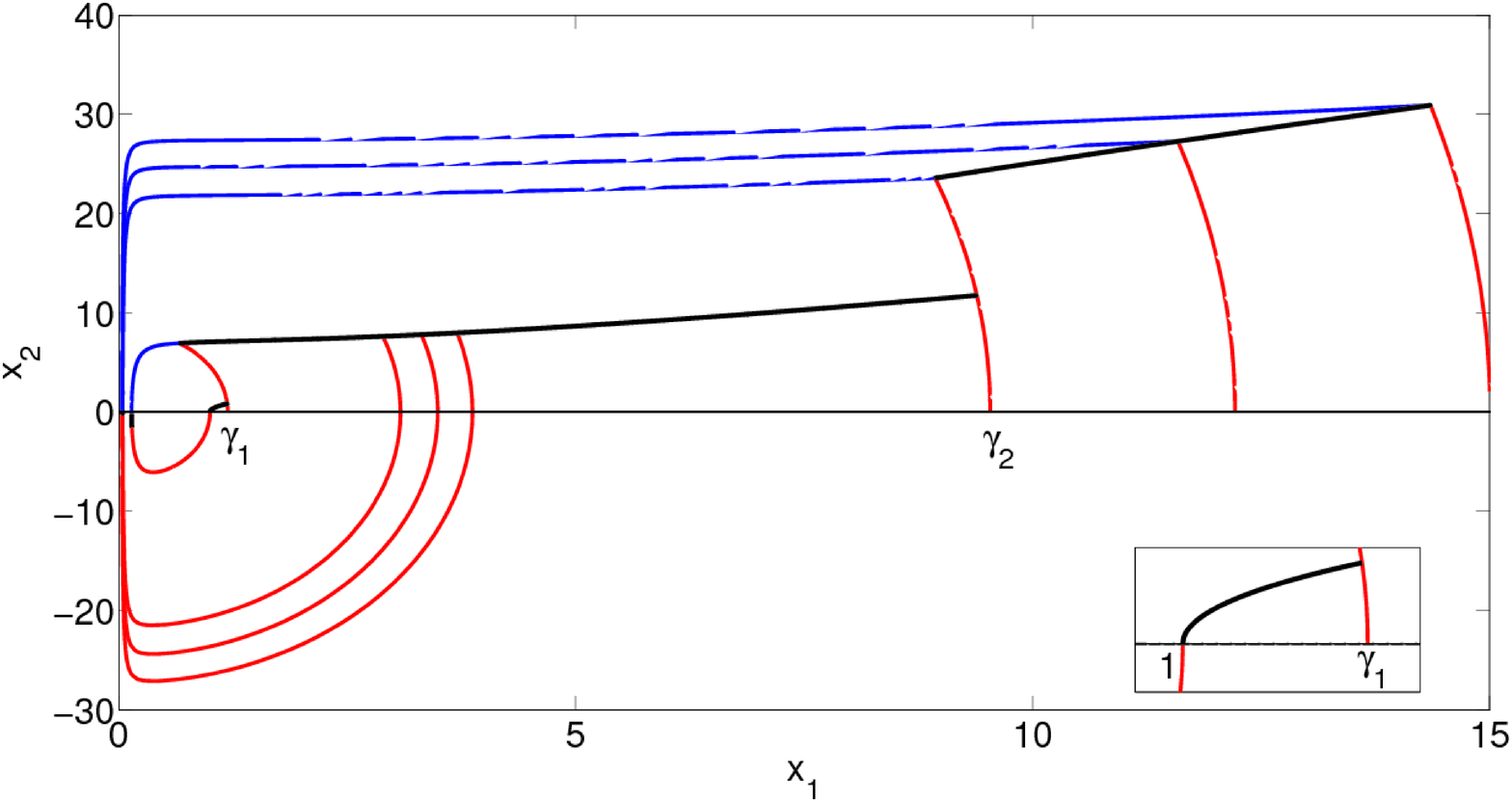}}
\end{tabular}
\caption{(a) Transfer times corresponding to zero, one, two and three turns
for $u_{1}=1, u_{2}=50, \gamma\in[1,15]$.(b) Switching curves (black solid curves)
and characteristic optimal trajectories starting from $(1,0)$.}%
\label{fig:fourmax}%
\end{figure}

In Fig.\ \ref{fig:times4} we plot the times $T_{0}, T_{1}, T_{2}$ and $T_{3}$ from Theorem \ref{prop:time}, corresponding to zero, one, two and three turns, for
$u_{1}=1, u_{2}=50$ and $\gamma\in[1,15]$. Again, for small $\gamma$ the
one-switching strategy is optimal and after some $\gamma=\gamma_{1}$ the
one-turn strategy becomes faster, but there is also a $\gamma=\gamma_{2}$
beyond which the two-turn strategy is optimal (up to the range of $\gamma$
plotted). The point $(\gamma_{2},0)$ thus belongs to the cut-locus of the one- and two-turn control sequences from $(1,0)$
since it can be reached with one or two turns in equal time. In Fig.\
\ref{fig:synthesis4} we plot the switching curves (black solid curves) along with
some characteristic optimal trajectories starting from $(1,0)$. For
$\gamma\geq\gamma_{2}$ the optimal trajectory makes an additional turn. This
is demonstrated by the three adjacent $Y$-segments (red dotted curves), which switch
close to $0$ to the corresponding $X$-segments (blue dashed curves), on a tiny
switching curve which is hardly seen. In turn, these trajectories switch on
the third switching curve on the upper quadrant to $Y$-segments (red dotted curves)
that hit the $x_{1}$-axis at the target points.

\section{CONCLUSION AND FUTURE WORK}

In this article we solved a time-optimal control problem related to frictionless atom cooling in harmonic traps. The results presented here can be immediately extended to the fast frictionless expansion of a two-dimensional Bose-Einstein condensate confined in a parabolic trapping potential \cite{Muga09}, with possible application in atom interferometry, and even to the implementation of a quantum dynamical microscope, a controlled expansion that allows to scale up an initial many-body state of an ultracold gas by a desired factor while preserving the quantum correlations of the initial state \cite{delCampo11}. 

This work has also been used to show how the third law of classical thermodynamics, known as unattainability principle, emerges from quantum thermodynamics \cite{Rezek11}. In a dynamical interpretation, this law states that absolute zero is unattainable, since the cooling rate from a thermal bath with falling temperature declines as well and approaches zero with an appropriate power of the temperature. The heat machine used to demonstrate this is a quantum refrigerator, the quantum analog of the classical Otto cycle, where the working medium is made up of particles in a harmonic (possibly repelling) potential instead of classical particles in a piston. The initial frequency coincides with the maximum allowed frequency as well as the strongest repelling frequency, so $u_1=u_2=1$. From Corollary \ref{uone} we see that the minimum time $T$ for the adiabatic-like cooling branch is given by equation (\ref{time_0}) in Theorem \ref{prop:time}. As the temperature approaches zero, $\omega_T\rightarrow 0$ and $\gamma=\sqrt{\omega_0/\omega_T}\rightarrow\infty$, so the cooling time approaches infinity logarithmically, $T\sim\ln\gamma$ \cite{Hoffmann_EPL11}.







\begin{thebibliography}{99}


\bibitem{Stefanatos11}
D. Stefanatos, H. Schaettler, and J.-S. Li, ``Minimum-time frictionless atom cooling in harmonic traps", {\it SIAM J. Control Optim.}, vol. 49, pp. 2440-2462, 2011.

\bibitem{Mabuchi05}H.~Mabuchi and N.~Khaneja, ``Principles and applications of control in quantum systems", {\it Int. J. Robust Nonlinear Control}, vol. 15, pp.~647--667, 2005.

\bibitem{Doherty00}A.~C. Doherty, S.~Habib, K.~Jacobs, H.~Mabuchi, and S.~M. Tan, ``Quantum feedback control and classical control theory", {\it Phys. Rev. A}, vol. 62, 012105, 2000.

\bibitem{James08}M.~R. James, H.~I. Nurdin, and I.~R. Petersen, ``$H^\infty$ control of linear quantum stochastic systems", {\it IEEE Trans. Automat. Control}, vol. 53, pp.~1787--1803, 2008.









\bibitem{Khaneja01}N.~Khaneja, R.~Brockett, and S.~J. Glaser, ``Time optimal control in spin systems", {\it Phys. Rev. A}, vol. 63, 032308, 2001.

\bibitem{D'Alessandro01}D.~D'Alessandro and M.~Dahleh, ``Optimal control of two-level quantum systems", {\it IEEE Trans. Automat. Control}, vol. 46, pp.~866--876, 2001.

\bibitem{Boscain02}U.~Boscain, G.~Charlot, J.-P.~ Gauthier, S.~Guerin, and H.-R.~ Jauslin, ``Optimal control in laser-induced population transfer for two- and three-level quantum systems", {\it J. Math. Phys.}, vol 43, pp.~2107--2132, 2002.

\bibitem{Stefanatos04}D.~Stefanatos, N.~Khaneja, and S.~J. Glaser, ``Optimal control of coupled spins in the presence of longitudinal and transverse relaxation", {\it Phys. Rev. A}, vol. 69, 022319, 2004.


\bibitem{Bonnard09}B.~Bonnard, M.~Chyba, and D.~Sugny, ``Time-minimal control of dissipative two-level quantum systems: The generic case", {\it IEEE Trans. Automat. Control}, vol. 54, pp.~2598--2610, 2009.


\bibitem{Stefanatos10}D.~Stefanatos and J.-S.~Li, ``Constrained minimum-energy optimal control of the dissipative Bloch equations", {\it Systems Control Lett.}, vol. 59, pp.~601--607, 2010.


\bibitem{Peirce88}A.~Peirce, M.~Dahleh, and H.~Rabitz, ``Optimal control of quantum mechanical systems: Existence, numerical approximations, and applications", {\it Phys. Rev. A}, vol. 37, pp.~4950--4964, 1988.


\bibitem{Khaneja05}N.~Khaneja, T.~Reiss, C.~Kehlet, T.~Schulte-Herbr\"{u}ggen, and S.~J. Glaser, ``Optimal control of coupled spin dynamics: Design of NMR pulse sequences by gradient ascent algorithms", {\it J. Magn. Reson.}, vol. 172, pp.~296--305, 2005.

\bibitem{Li09}J.-S.~Li, J.~Ruths, and D.~Stefanatos, ``A pseudospectral method for optimal control of open quantum systems", {\it J. Chem. Phys.}, vol. 131, 164110, 2009.






\bibitem{Wieman99}C.~E. Wieman, D.~E. Pritchard, and D.~J Wineland, ``Atom cooling, trapping, and quantum manipulation", {\it Rev. Mod. Phys.}, vol. 71, pp.~S253--S262, 1999.

\bibitem{Cirac04}J.~I. Cirac and P.~Zoller, ``New frontiers in quantum information with atoms and ions", {\it Physics Today}, vol. 57, pp.~38-44, 2004.

\bibitem{Lewis69}
H.~R. Lewis and W.~B. Riesenfeld, ``An exact quantum theory of the time-dependent harmonic oscillator and of a charged particle in a time-dependent electromagnetic field", {\it J. Math. Phys.}, vol. 10, pp.~1458--1473, 1969.

\bibitem {Leanhardt03}A.~E. Leanhardt, T.~A. Pasquini, M.~Saba, A.~Schirotzek,
Y.~Shin, D.~Kielpinski, D.~E. Pritchard, and W.~Ketterle, ``Adiabatic and evaporative cooling of Bose-Einstein condensates below 500 picokelvin", Science, vol. 301, pp.~1513--1515, 2003.

\bibitem {Bize05}S.~Bize, P.~Laurent, M.~Abgrall, H.~Marion, I.~Maksimovic, L.~Cacciapuoti, J.~Grünert, C.~Vian, F.~Pereira dos Santos, P.~Rosenbusch, P.~Lemonde, G.~Santarelli, P.~Wolf, A.~Clairon, A.~Luiten, M.~Tobar, and C.~Salomon, ``Cold atom clocks and applications", {\it J. Phys. B: At. Mol. Opt. Phys.}, vol. 38, pp.~S449--S468, 2005.


\bibitem {Salamon09}P.~Salamon, K.~H. Hoffmann, Y.~Rezek, and R.~Kosloff, ``Maximum work in minimum time from a conservative quantum system", {\it Phys.
Chem. Chem. Phys.}, vol. 11, pp.~1027-1032, 2009.


\bibitem {Chen10}X.~Chen, A.~Ruschhaupt, S.~Schmidt, A.~del Campo, D.~Gu\'{e}ry-Odelin, and J.~G. Muga, ``Fast optimal frictionless atom cooling in harmonic traps: Shortcut to adiabaticity", {\it Phys. Rev. Lett.}, vol. 104, 063002, 2010.

\bibitem {Stefanatos10PRA}D.~Stefanatos, J.~Ruths, and J.-S.~Li, ``Frictionless atom cooling in harmonic traps: A time-optimal approach", {\it Phys. Rev. A}, vol. 82, 063422, 2010.

\bibitem{Hoffmann_EPL11}
K.-H. Hoffmann, P. Salamon, Y. Rezek, and R. Kosloff, ``Time-optimal controls for frictionless cooling in harmonic traps", {\it Europhys. Lett.}, vol. 96, 60015, 2011.

\bibitem{Rezek11}
Y. Rezek, {\it Heat Machines and Quantum Systems: Towards the Third Law}, PhD Thesis, Hebrew University of Jerusalem; 2011.

\bibitem{Merzbacher98}
E.~Merzbacher, {\it Quantum Mechanics}, John Wiley and Sons, New York; 1998.



\bibitem{Pontryagin}L.~S. Pontryagin, V.~G. Boltyanskii, R.~V. Gamkrelidze, and
E.~F. Mishchenko, {\it The Mathematical Theory of Optimal Processes}, Interscience
Publishers, New York; 1962.


\bibitem {Su1}H.~J. Sussmann, ``The structure of time-optimal trajectories for
single-input systems in the plane: The $C^{\infty}$ nonsingular case",
{\it SIAM J. Control Optim.}, vol. 25, pp.~433--465, 1987.


\bibitem{Jurdjevic97}V.~Jurdjevic, {\it Geometric Control Theory}, Cambridge University Press, Cambridge; 1997.


\bibitem {Muga09}J.~G. Muga, X.~Chen, A.~Ruschhaupt, and D.~Gu\'{e}ry-Odelin, ``Frictionless dynamics of Bose-Einstein condensates under fast trap variations",
{\it J. Phys. B: At. Mol. Opt. Phys.}, vol. 42, 241001, 2009.

\bibitem{delCampo11}A.~del Campo, ``Frictionless quantum quenches in ultracold gases: A quantum dynamical microscope", {\it Phys. Rev. A}, vol. 84, 031606(R), 2011.































\end{thebibliography}
\end{document}